\documentclass[10pt]{amsart}

\usepackage{amsfonts, amssymb, amsmath, enumerate}

\numberwithin{equation}{section}

\DeclareMathOperator{\E}{\mathbb{E}}

\DeclareMathOperator*{\re}{Re}

\def \C {\mathbb{C}}
\def \N {\mathbb{N}}
\def \P {\mathbb{P}}
\def \R {\mathbb{R}}

\def \E {\mathbb{E}}
\def \D {\mathbb{D}}

\def \HH {\mathcal{H}}

\def \vol {{\rm vol}}
\def \Id {{\rm Id}}

\def \etc {,\ldots,}

\newtheorem{theorem}{Theorem}[section]
\newtheorem{proposition}[theorem]{Proposition}
\newtheorem{corollary}[theorem]{Corollary}
\newtheorem{lemma}[theorem]{Lemma}

\theoremstyle{remark}

\begin{document}

\title{Maximal sections of the unit ball of $\ell^n_p(\C)$ for $p > 2$}

\author{Jacek Jakimiuk and Hermann K\"onig}
\thanks{J.J.’s research was supported by the National Science
Centre, Poland, grant 2018/31/D/ST1/0135.}

\keywords{Volume, hyperplane sections, $\ell_p$-ball, random variables}
\subjclass[2020]{Primary: 52A38, 52A40 Secondary: 46B07, 60F05}

\begin{abstract}
Eskenazis, Nayar and Tkocz \cite{ENT} have shown recently some resilience of Ball's celebrated cube slicing theorem, namely its analogue in $\ell^n_p$ for large $p$. We show that the complex analogue, i.e. resilience of the polydisc slicing theorem proven by Oleszkiewicz and Pelczy\'nski, holds for large $p$ and small $n$, but does not hold for any $p > 2$ and large $n$.
\end{abstract}

\maketitle

\section{Introduction and main results}

Calculating the volume of sections of convex sets by subspaces is not an easy problem, even for classical convex bodies. The Busemann-Petty problem, see e.g. Koldobsky \cite{K}, and the hyperplane conjecture, cf. Milman, Pajor \cite{MiP}, gave an impetus to study hyperplane sections of convex bodies in detail. In a celebrated  paper Ball \cite {B} proved that the hyperplane section of the $n$-cube perpendicular to $a^{(2)} = \frac 1 {\sqrt 2} (1,1,0 \etc 0) \in S^{n-1} \subset \R^n$ has maximal volume among all hyperplane sections. Using the Brascamp-Lieb inequality, in the paper \cite{B2} he generalized this result to $k$-codimensional sections of the $n$-cube, $1 \le k \le n-1$. Earlier Hadwiger \cite{Ha} and Hensley \cite{He} had shown independently of one another that coordinate hyperplanes e.g. orthogonal to $a^{(1)} = (1,0 \etc 0) \in S^{n-1}$ yield the minimal $(n-1)$-dimensional cubic sections. Vaaler \cite{V} generalized this to $k$-codimensional cubic sections, $1 \le k \le n-1$. Chasapis, Nayar and Tkocz \cite{CNT} proved a dimension-free stability result for these upper and lower bounds of hyperplane cubic sections. Nayar and Tkocz gave an excellent survey on sections and projections of convex bodies, see \cite{NT}.\\

Meyer and Pajor \cite{MP} found the extremal sections of the $\ell_p^n$ balls $B_p^n$: they proved that the normalized volume of sections of $B_p^n$ by a fixed $k$-codimensional subspace is monotone increasing in $p$. This implies that coordinate hyperplanes provide the minimal sections for $2 \le p < \infty$, as for $p=\infty$, and the maximal sections for $1 \le p \le 2$. Their result is also valid in the complex case $B_p^n(\C)$. The minimal hyperplane sections of $B_1^n$ are those orthogonal to a main diagonal, e.g. to $a^{(n)} = \frac 1 {\sqrt n}(1 \etc 1) \in S^{n-1}$, see also \cite{MP}. Koldobsky \cite{K} extended this to
$1 \le p \le 2$. His argument also covers the range $0 < p < 1$. This left open the case of the maximal hyperplane sections of $B_p^n$ for $2 < p < \infty$. The situation there is more complicated, since then the maximal hyperplane may depend as well on $p$ as on the dimension $n$. Oleszkiewicz \cite {O} proved that Ball's result does not transfer to the balls $B_p^n$ if $2 < p < p_0 \simeq 26.265$: the intersection of the hyperplane perpendicular to $a^{(n)}$ has larger volume than the one orthogonal to $a^{(2)}$, for sufficiently large dimensions $n$. On the other hand, recently Eskenazis, Nayar and Tkocz \cite{ENT} proved that Ball's result is stable for $\ell_p^n$ and very large $p$: $(a^{(2)})^\perp \cap B_p^n$ is the maximal hyperplane section of $B_p^n$ for all dimensions, provided that $p_1 := 10^{15} < p < \infty$. In the sequel, we may refer to this phenomenon as ``resilience of cubic sections". The paper K\"onig \cite{Ko} further studies the case $p_0 < p < \infty$. \\

The complex analogue of Ball's result was shown by Oleszkiewicz and Pelczy\'nski \cite {OP}: In the case of the polydisc $B_\infty^n(\C)$, i.e. the complex $\ell_\infty^n$-ball, the complex hyperplane orthogonal to $a^{(2)}$ still yields the maximal complex hyperplane section of $B_\infty^n(\C)$. For volume considerations $\C^n$ is identified with $\R^{2n}$. However, for the normalized polydisc $\tilde{B}_\infty^n(\C) = \frac 1 {\pi} B_\infty^n(\C)$ we have
$$\vol_{2(n-1)}\left( \left( a^{(2)} \right)^\perp \cap \tilde{B}_\infty^n(\C) \right) = \lim_{n \to \infty} \vol_{2(n-1)}\left( \left( a^{(n)} \right)^\perp \cap \tilde{B}_\infty^n(\C) \right),$$
so that $\left( a^{(n)} \right)^\perp \cap \tilde{B}_\infty^n(\C)$ barely misses to have maximal volume for large dimension $n$. In a stability result for polydisc slicing shown recently by Glover, Tkocz and Wyczesany \cite {GTW}, estimating the volume $\vol_{2(n-1)}\left( a^\perp \cap B_\infty^n(\C) \right)$ for unit vectors $a$ close to $a^{(2)}$ requires a fourth order term $||a||_4$ in addition to $\left| a-a^{(2)} \right|$, where $| \cdot |$ denotes the Euclidean norm. Both facts indicate that, in contrast to the real case, ``resilience of polydisc slicing" may be violated for the complex $\ell_p^n$-unit ball $B_p^n(\C)$ for all $2 < p < \infty$ and large dimensions $n$. We prove exactly this, also giving an estimate for the dimensions $n$ needed to have
\begin{align*}
    \vol_{2(n-1)}\left( \left( a^{(n)} \right)^\perp \cap B_\infty^n(\C) \right) > \vol_{2(n-1)}\left( \left( a^{(2)} \right)^\perp \cap B_\infty^n(\C) \right).
\end{align*}
Let $1 \le p \le \infty$, $n \in \N$ and $a \in \C^n$, $|a| = 1$. We use the notation
$$A_{n,p}(a) := \frac{\vol_{2(n-1)}\left( a^\perp \cap B_p^n(\C) \right)}{\vol_{2(n-1)}\left( B_p^{n-1}(\C) \right)} $$
for the normalized section volume of $B_p^n(\C)$ perpendicular to $a$.

\begin{theorem}\label{th1}
Let $2 < p < \infty$, $n \in \N$ and $a \in \C^n$, $|a| = 1$. Then for every $2 < p < \infty$ there is a constant $N(p)>0$ such that for all $n \ge N(p)$,
$$A_{n,p}\left( a^{(n)} \right) > A_{n,p}\left( a^{(2)} \right). $$
In fact, this holds with $N(p) = \frac 5 2 p$ if $p \ge 9$ and with $N(p) = p$ if $p \ge 140$. We have for all $2 < p < \infty$ that
$$\lim_{n \to \infty} A_{n,p}\left( a^{(n)} \right) = 2 \frac{\Gamma\left( 1+ \frac 2 p \right)^2}{\Gamma\left( 1 + \frac 4 p \right)} > A_{n,p}\left( a^{(2)} \right) = 2^{1-\frac 2 p}. $$
\end{theorem}

There is a dichotomy between non-resilience of polydisc slicing for large dimensions $n$ and resilience of polydisc slicing for small dimensions $n$ and large $p$: for large $p$ and relatively low dimensions $n$ the fourth order term in the estimation of the volume $\vol_{2(n-1)}\left( a^\perp \cap B_\infty^n(\C) \right)$ by Glover, Tkocz and Wyczesany is bounded from below and thus is negligible in the stability estimation of \cite{ENT}. In this situation the proof of  Eskenazis's, Nayar's and Tkocz's theorem can be adjusted to the complex case. We prove a complex analogue of the Eskenazis-Nayar-Tkocz theorem for dimensions $n$ low relative to $p$, i.e. if $n \leq cp$ for some universal constant $c$.

\begin{theorem}\label{th2}
    Let $p > p_2 := 10^{56}$ and $n < N(p) := \frac{p}{1520}$. Then for $a \in \C^n$, $|a| = 1$ the inequality
    $$A_{n,p}(a) \leq A_{n,p}\left( a^{(2)} \right) = 2^{1 - \frac{2}{p}}$$
    holds.
\end{theorem}

{\it Remarks.} i) The constants in Theorem \ref{th2} are far from being optimal. The proof of Theorem \ref{th2} is similar to the proof of Eskenazis, Nayar, Tkocz \cite{ENT}, Theorem 1. It uses the stability result by Glover, Tkocz, Wyczesany \cite{GTW}, Theorem 1, and proceeds by induction on the dimension. \\

ii) To prove Theorem \ref{th1}, we use a formula for $A_{n,p}(a)$ similar to the one used in the real case by Eskenazis, Nayar and Tkocz \cite {ENT}, Proposition 6, the central limit theorem and direct error estimates for $A_{n,p}\left( a^{(n)} \right) - \lim_{m \to \infty} A_{m,p}\left( a^{(m)} \right).$\\

iii) Concerning the restriction $p \ge 9$ in Theorem \ref{th1}, we remark that \\
$\lim_{p \searrow 2} 2 \frac{\Gamma\left( 1 + \frac 2 p \right)^2}{\Gamma\left( 1 + \frac 4 p \right)} =
\lim_{p \searrow 2} 2^{1 - \frac 2 p} = 1$, and the statement $A_{n,p}\left( a^{(n)} \right) > A_{n,p}\left( a^{(2)} \right)$ for all $n \ge N(p)$ in Theorem \ref{th1} also holds for $2 < p < 9$ when $N(p) \le \frac c {p-2}$ is satisfied for some absolute constant $c>0$. \\

In section 2 we give some preliminary results, in section 3 we verify Theorem \ref{th1}, and in section 4 we prove Theorem \ref{th2}. \\

\section{Preliminary results}

We start with a complex analogue of Proposition 6 of Eskenazis, Nayar and Tkocz \cite {ENT}. Using the same notation $A_{n,p}(a)$ as in Theorem \ref{th1}, we have

\begin{proposition}\label{prop1}
Let $1 \le p < \infty$, $n \in \N$ and $(\xi_j)_{j=1}^n$ be i.i.d. random vectors uniformly distributed on the sphere $S^3 \subset \R^4$ and $(R_j)_{j=1}^n$ be i.i.d. random variables with density $c_p^{-1} t^{p+1} \exp(-t^p)$ on $[0,\infty)$, $c_p := \frac 1 p \Gamma\left( 1+\frac 2 p \right)$, independent of the  $(\xi_j)_{j=1}^n$. Then for any $a = (a_j)_{j=1}^n \in S^{n-1} \subset \R^n$ we have
$$A_{n,p}(a) = \Gamma\left( 1+\frac 2 p \right) \ \E_{\xi,R} \left| \sum_{j=1}^n a_j R_j \xi_j \right|^{-2}. $$
\end{proposition}

\begin{proof}
    We shall divide the proof into several steps, in each of them we shall obtain a formula closer to the one we are aiming for.

    \textit{Step 1.} Corollary 4.4 in Chasapis, Nayar, Tkocz \cite{CNT} with $k = 2$, $\| \cdot \| = | \cdot |$, $K = \D$, $H = a^{\perp}$, $A = cB_p^n(\C)$, where $c$ is such that $\vol_{2n}(A) = 1$, and $X$ uniformly distributed on $B_p^n(\C)$ gives us
    \begin{align*}
        \vol_{2(n-1)}\left( cB^n_p(\C) \cap a^{\perp} \right) &= \lim_{q \nearrow 2}\frac{2-q}{2\pi}\E\left| \langle a, cX \rangle_{\C} \right|^{-q},
    \end{align*}
    where $\langle \cdot, \cdot \rangle_{\C}$ denotes the complex scalar product. Thus
    \begin{align*}
        \frac{\vol_{2(n-1)}\left( B^n_p(\C) \cap a^{\perp} \right)}{\vol_{2n}\left( B^n_p(\C) \right)} &= c^{2n}\vol_{2(n-1)}\left( B^n_p(\C) \cap a^{\perp} \right) = c^2\vol_{2(n-1)}\left( cB^n_p(\C) \cap a^{\perp} \right) \\ &= c^2\lim_{q \nearrow 2}\frac{2-q}{2\pi}\E\left| \langle a, cX \rangle_{\C} \right|^{-q} \\ &= \lim_{q \nearrow 2}\frac{2-q}{2\pi}c^{2-q}\E\left| \langle a, X \rangle_{\C} \right|^{-q} \\ &= \lim_{q \nearrow 2}\frac{2-q}{2\pi}\E\left| \langle a, X \rangle_{\C} \right|^{-q}.
    \end{align*}

    \textit{Step 2.} Let $Y_1, \ldots, Y_n$ be i.i.d. complex-valued random variables with density $e^{-\beta_p|z|^p}$, where $\beta_p = \left( \pi\Gamma\left( 1 + \frac 2 p \right) \right)^{\frac p 2}$. Denote $Y = (Y_1, \ldots, Y_n)$ and $S = \|Y\|_p$. By applying Proposition 9.3.3 in Prochno, Th\"ale, Turchi \cite{PTT} with $K = B^n_p(\C)$ and $Z = Y$ treated as vector in $\R^{2n}$, we obtain that $U^{\frac{1}{2n}}\frac{Y}{S}$ is uniformly distributed on $B^n_p(\C)$ and $S$ is independent of $\frac{Y}{S}$, where $U$ is uniformly distributed on $[0, 1]$ and independent of $Y$. Take $D$ uniformly distributed on $\D$ and independent of all other vectors, then clearly $|D|^2 \sim U$ and we obtain that $|D|^{\frac{1}{n}}\frac{Y}{S}$ is uniformly distributed  on $B^n_p(\C)$. Then
        \begin{align*}
            \E\left| \langle a, X \rangle_{\C} \right|^{-q} &= \E\left| \left\langle a, |D|^{\frac{1}{n}}\frac{Y}{S} \right\rangle_{\C} \right|^{-q} = \E\left| \frac{|D|^{\frac{1}{n}}}{S}\sum_{j=1}^na_jY_j \right|^{-q} \\ &= \frac{\E S^{-q}}{\E S^{-q}}\E|D|^{-\frac{q}{n}}\E\left| \sum_{j=1}^na_j\frac{Y_j}{S} \right|^{-q} = \frac{\E|D|^{-\frac{q}{n}}}{\E S^{-q}}\E\left| \sum_{j=1}^na_jY_j \right|^{-q},
        \end{align*}
    where we used in the second equality that $Y$ and its complex conjugate $\left( \overline{Y_1}, \ldots, \overline{Y_n} \right)$ have the same distribution and in the last equality that $S$ and $\frac{Y}{S}$ are independent. Applying this to the result of Step 1 yields
    \begin{equation}\label{formula21}
        \frac{\vol_{2(n-1)}\left( B^n_p(\C) \cap a^{\perp} \right)}{\vol_{2n}\left( B^n_p(\C) \right)} = \lim_{q \nearrow 2}\frac{2-q}{2\pi}\frac{\E|D|^{-\frac{q}{n}}}{\E S^{-q}}\E\left| \sum_{j=1}^na_jY_j \right|^{-q}.
    \end{equation}
    Take $a = a^{(1)} = (1, 0, \ldots, 0)$, then the above equality assumes the form
    \begin{equation}\label{formula22}
        \frac{\vol_{2(n-1)}\left( B^{n-1}_p(\C) \right)}{\vol_{2n}\left( B^n_p(\C) \right)} = \lim_{q \nearrow 2}\frac{2-q}{2\pi}\E|Y_1|^{-q}\lim_{q \nearrow 2}\frac{\E|D|^{-\frac{q}{n}}}{\E S^{-q}}.
    \end{equation}
    It follows from Lemma 4.3 in \cite{CNT} with $f$ there being the density of $Y_1$ that
    \begin{align*}
        \lim_{q \nearrow 2}\frac{2-q}{2\pi}\E|Y_1|^{-q} = 1.
    \end{align*}
    Substituting it into \eqref{formula22} and dividing \eqref{formula21} by \eqref{formula22} proves that
    \begin{equation}\label{formula2}
        A_{n, p}(a) = \lim_{q \nearrow 2}\frac{2-q}{2\pi}\E\left| \sum_{j=1}^na_jY_j \right|^{-q}.
    \end{equation}

    \textit{Step 3.} Let $g(t) = e^{-\beta_pt^p}$ for $t \in \R$. Then $Y_1$ has density $g(|z|)$ and $R_1' := \beta_p^{-\frac 1 p}R_1$ has density $-\pi t^2g'(t)\mathbf{1}_{t > 0}$. Let also $D_1, \ldots, D_n$ be i.i.d random variables uniformly distributed on $\D$ and independent of $R_1, \ldots, R_n$. We shall prove that $Y_1$ has the same distribution as $R_1'D_1$. Both variables are rotationally invariant in the plane, hence it suffices to check that their magnitudes have the same distribution. For $s > 0$ we have
    \begin{align*}
        \P(|Y_1| > s) &= \int_{|z| > s}g(|z|)dz = 2\pi\int_s^{\infty}tg(t)dt = -2\pi\left( \frac{s^2}{2}g(s) + \int_s^{\infty}\frac{t^2}{2}g'(t)dt \right) \\ &= -\pi\int_s^{\infty}\left(t^2 - s^2 \right)g'(t)dt = -\pi\int_s^{\infty}\left( 1 - \frac{s^2}{t^2} \right)t^2g'(t)dt \\ &= \int_0^{\infty}\P\left( |D_1| > \frac{s}{t} \right)\left( -\pi t^2g'(t) \right)dt \\ &= \P\left( |D_1| > \frac{s}{R_1'} \right) = \P(|D_1|R_1' > s),
    \end{align*}
    where in the second equality we passed to the polar coordinates. The fact that $Y_j$ is equidistributed to $R'_jD_j$ lets us rewrite \eqref{formula2} as
    \begin{equation}\label{formula3}
        A_{n. p}(a) = \lim_{q \nearrow 2}\beta_p^{\frac q p}\frac{2-q}{2\pi}\E\left| \sum_{j=1}^na_jR_jD_j \right|^{-q}.
    \end{equation}

    \textit{Step 4.} Our final aim is to remove the limit from the formula \eqref{formula3}. By Proposition 4 of K\"onig and Kwapie\'n \cite{KK} we have
    \begin{align*}
        \E\left| \sum_{j=1}^na_j\xi_j\right|^{-q} = \frac{2-q}{2}\E\left| \sum_{j=1}^na_jD_j \right|^{-q}
    \end{align*}
    for $q < 2$. Conditioning on $R_j$, passing to the limit $q \nearrow 2$ using the last equality, and substituting the value of $\beta_p = \left(\pi\Gamma\left( 1 + \frac 2 p \right) \right)^{\frac p 2}$ yields the claim of Proposition \ref{prop1}.
\end{proof}

\textit{Remark.} For $p = \infty$ we have
\begin{align*}
    A_{n,\infty}(a) := \frac{\vol_{2(n-1)}\left( a^\perp \cap B_{\infty}^n(\C) \right)}{\vol_{2(n-1)}\left( B_{\infty}^{n-1}(\C) \right)} = \E_{\xi}\left| \sum_{j=1}^na_j\xi_j \right|^{-2},
\end{align*}
which was proved by Brzezinski \cite{Br} (Proposition 3.2). \\

One part of the proof of Theorem \ref{th1} is based on a second formula for $A_{n,p}(a)$ which is derived from Proposition \ref{prop1}. Here $J_0$ denotes the standard Bessel function of order $0$.

\begin{proposition}\label{prop2}
Let $1 \le p < \infty$, $n \in \N$ and $a = (a_j)_{j=1}^n \in S^{n-1} \subset \R^n$. Then
$$A_{n,p}(a) = \Gamma\left( 1+\frac 2 p \right) \frac 1 2 \int_0^\infty \prod_{j=1}^n \gamma_p(a_j s) \ s \ ds, $$
where
$$\gamma_p(s) := \frac 2 {\Gamma\left( 1+\frac 2 p \right)} \int_0^\infty J_0(sr) \exp(-r^p) \ r \ dr. $$
\end{proposition}

\begin{proof}
Let $j_1(s):= 2 \frac {J_1(s)} s$, where $J_1$ denotes the Bessel function of order $1$. Then $\lim_{s \searrow 0} j_1(s) = 1$ and for any $t>0$
\begin{equation}\label{eq2.1}
\int_0^\infty j_1(ts) \ s \ ds = \frac 2 {t^2} \int_0^\infty J_1(u) du = \frac 2 {t^2}
\end{equation}
since $J_0' = -J_1$. Thus with the notation from Proposition \ref{prop1}
\begin{align*}
\E_{\xi,R} \left| \sum_{j=1}^n a_j R_j \xi_j \right|^{-2} &= \frac 1 2 \E_{\xi,R} \int_0^\infty j_1\left( \left| \sum_{j=1}^n a_j R_j \xi_j \right| s \right) \ s \ ds \\
&=\frac 1 2 \int_0^\infty \E_{\xi,R} \ j_1\left( \left| \sum_{j=1}^n a_j R_j \xi_j \right| s \right) \ s \ ds.
\end{align*}
Interchanging $\E_{\xi,R}$ and  $\int_0^\infty$ has to be justified, since \eqref{eq2.1} is only a conditionally convergent Riemann integral.
The verification is similar as in the proof of Proposition 3.2 (b) in K\"onig, Rudelson \cite{KR}. The argument is as follows: Let $N>0$. For finite intervals $[0,N]$ clearly we may interchange $\int_0^N$ and $\E_{\xi,R}$. Therefore it suffices to show
\begin{equation}\label{eq2.1a}
\lim_{N \to \infty} \E_{\xi,R} \int_N^\infty j_1\left( \left| \sum_{j=1}^n a_j R_j \xi_j \right| s \right) \ s \ ds = 0
\end{equation}
and
\begin{equation}\label{eq2.1b}
\lim_{N \to \infty} \int_N^\infty \E_{\xi,R} \ j_1 \left( \left| \sum_{j=1}^n a_j R_j \xi_j \right| s \right) \ s \ ds = 0 \ .
\end{equation}
Using that $J_0' = -J_1$, we find, if $\left| \sum_{j=1}^n a_j R_j \xi_j \right| \ne 0$, that
\begin{align*}
\left| \int_N^\infty j_1 \left( \left| \sum_{j=1}^n a_j R_j \xi_j \right| s \right) \ s \ ds \right| & = 2 \frac{\left| \int_N^\infty J_1\left( \left| \sum_{j=1}^n a_j R_j \xi_j \right| s \right) \ s \ ds \right| }{\left| \sum_{j=1}^n a_j R_j \xi_j \right|} \\
& = 2 \frac{\left|J_0 \left(\left| \sum_{j=1}^n a_j R_j \xi_j \right| \ N \right)\right|}{\left| \sum_{j=1}^n a_j R_j \xi_j \right|^2} \le \frac 2 {\left| \sum_{j=1}^n a_j R_j \xi_j \right|^2} \ ,
\end{align*}
which is integrable with respect to $(\xi,R)$ by Proposition \ref{prop1}, independently of $N>0$. Further
$\lim_{N \to \infty} J_0 \left(\left| \sum_{j=1}^n a_j R_j \xi_j \right| \ N \right) = 0$, if $\left|\sum_{j=1}^n a_j R_j \xi_j\right| \ne 0$. Thus \eqref{eq2.1a} follows using Lebesgue's dominated convergence theorem. As for \eqref{eq2.1b}, we have
\begin{align*}
\left|- \frac d {ds} \frac{J_0 \left(\left| \sum_{j=1}^n a_j R_j \xi_j \right| \ s \right)}{\left|\sum_{j=1}^n a_j R_j \xi_j \right|^2} \right| & = \left|\frac{J_1 \left(\left| \sum_{j=1}^n a_j R_j \xi_j \right| \ s \right)}{\left|\sum_{j=1}^n a_j R_j \xi_j \right|} \right| \\
& =  \frac 1 2 \left|j_1 \left(\left| \sum_{j=1}^n a_j R_j \xi_j \right| \ s \right) \ s \right| \le \frac 1 {\left|\sum_{j=1}^n a_j R_j \xi_j \right|}
\end{align*}
with
$$\E_{\xi,R} \left|\frac d {ds} \frac{J_0 \left(\left| \sum_{j=1}^n a_j R_j \xi_j \right| \ s \right)}{\left|\sum_{j=1}^n a_j R_j \xi_j \right|^2} \right| \le \E_{\xi,R} \frac 1 {\left|\sum_{j=1}^n a_j R_j \xi_j \right|}  \le \left( \E_{\xi,R} \frac 1 {\left|\sum_{j=1}^n a_j R_j \xi_j \right|^2}  \right)^{\frac 1 2} $$
being bounded independently of $s>0$. Therefore
\begin{align*}
-2 \frac d {ds} \left( \E_{\xi,R} \ \frac{J_0 \left(\left| \sum_{j=1}^n a_j R_j \xi_j \right| \ s \right)}{\left|\sum_{j=1}^n a_j R_j \xi_j \right|^2} \right) & =
-2 \E_{\xi,R} \left( \frac d {ds} \frac{J_0 \left(\left| \sum_{j=1}^n a_j R_j \xi_j \right| \ s \right)}{\left|\sum_{j=1}^n a_j R_j \xi_j \right|^2} \right) \\
& = \E_{\xi,R} \ j_1 \left( \left| \sum_{j=1}^n a_j R_j \xi_j \right| s \right) \ s
\end{align*}
and hence
\begin{align*}
\left|\int_N^\infty \E_{\xi,R} \ j_1 \left( \left| \sum_{j=1}^n a_j R_j \xi_j \right| s \right) \ s \ ds \right| & = 2 \left|\E_{\xi,R} \left( \frac{J_0 \left(\left| \sum_{j=1}^n a_j R_j \xi_j \right| \ N \right)}{\left|\sum_{j=1}^n a_j R_j \xi_j \right|^2} \right) \right| \\
& \le 2 \E_{\xi,R} \left( \frac 1 {\left|\sum_{j=1}^n a_j R_j \xi_j \right|^2} \right)
\end{align*}
is bounded independent of $N$, implying \eqref{eq2.1b} by Lebesgue's dominated convergence theorem, since pointwise
$\lim_{N \to \infty} J_0 \left(\left| \sum_{j=1}^n a_j R_j \xi_j \right| \ N \right) = 0$ for $\left|\sum_{j=1}^n a_j R_j \xi_j\right| \ne 0$. \\

Let $e \in S^3$ be a fixed vector and $m$ denote the normalized Lebesgue surface measure on $S^3$. Then for any $t \in \R$, cf. \cite {KR} ,
\begin{equation}\label{eq2.2}
\int_{S^3} \exp\left( i t \langle e,u \rangle \right) \ dm(u) = j_1(t).
\end{equation}
This implies for $(b_j)_{j=1}^n \in \R^n$
$$\prod_{j=1}^n j_1(b_j s) = \int_{(S^3)^n} \exp\left( i s \left\langle e, \sum_{j=1}^n b_j u_j \right\rangle \right) \ \prod_{j=1}^n dm(u_j) \ , $$
which holds for every $e \in S^3$. Averaging over $e \in S^3$, we find from \eqref{eq2.2}
$$\int_{S^3} \exp\left( i s \left\langle e, \sum_{j=1}^n b_j u_j \right\rangle \right) dm(e) = j_1 \left( \left| \sum_{j=1}^n b_j u_j \right| s \right) $$
and hence
$$ \prod_{j=1}^n j_1(b_j s) = \int_{(S^3)^n} j_1\left( \left| \sum_{j=1}^n b_j u_j \right| s \right) \ \ \prod_{j=1}^n dm(u_j) = \E_\xi \ j_1\left( \left| \sum_{j=1}^n b_j \xi_j \right| s \right) \ . $$
Hence, using the independence of the $R_j$
\begin{align*}
\E_{\xi,R} \left| \sum_{j=1}^n a_j R_j \xi_j \right|^{-2} &= \frac 1 2 \int_0^\infty \E_R\left (\prod_{j=1}^n j_1(a_j R_j s) \right) \ s \ ds \\
&= \frac 1 2 \int_0^\infty \prod_{j=1}^n \E_{R_j}( j_1(a_j R_j s)) \ s \ ds.
\end{align*}
To calculate $\E_{R_1}(j_1(R_1 s))$, we use that $\frac d {dx} (J_1(x) x) = J_0(x) x$, cf. Watson \cite{W}, 3.13, so that
$\frac d {dr} \left( \frac {J_1(sr)} s r \right) = J_0(sr) r$, and hence integration by parts yields
\begin{align*}
\E_{R_1}(j_1(R_1 s)) &= c_p^{-1} \int_0^\infty j_1(sr) r^{p+1} \exp(-r^p) dr \\ &= 2 c_p^{-1} \int_0^\infty \frac{J_1(sr)} s r \ r^{p-1} \exp(-r^p) dr \\
&= \frac 2 p c_p^{-1} \int_0^\infty J_0(sr) r \exp(-r^p) dr = \gamma_p(s).
\end{align*}
with $\frac 2 p c_p^{-1} = \frac 2 {\Gamma\left( 1+\frac 2 p \right)}$. This and Proposition \ref{prop1} implies
$$A_{n,p}(a) = \Gamma\left( 1+\frac 2 p \right) \frac 1 2 \int_0^\infty \prod_{j=1}^n \gamma_p(a_js) \ s \ ds \ . $$
\end{proof}

{\it Remark.} Proposition \ref{prop2} for $1 \le p \le 2$ can also be derived from Theorem 2 of Koldobsky, Zymonopoulou \cite {KZ} which has the form
$$\vol_{2(n-1)} \left( a^\perp \cap B_p^n(\C) \right) = c_{p,n} \int_0^\infty \prod_{j=1}^n f(|a_j| s) \ s \ ds \ , \ a = (a_j)_{j=1}^n \in \C^n, \ \sum_{j=1}^n |a_j|^2 = 1 \ , $$
where $f(s) = \int_{\R^2} \exp(-(u^2+v^2)^{\frac p 2}) \exp(-i u s) \ du \ dv$ and $c_{p,n} = \frac 1 {2 \pi} \frac 1 {\Gamma \left(1+\frac{2n-2} p \right)}$. Polar integration with $u = r \cos(\phi), v = r \sin(\phi)$ yields, using a standard formula for the Bessel function $J_0$,
\begin{align*}
f(s) & = \int_0^\infty \exp(-r^p) \ r \ \left(\int_0^{2 \pi} \cos(s r \cos(\phi)) \ d \phi \right) \ dr \\
& = 2 \pi \int_0^\infty J_0(s r) \ \exp(-r^p) \ r \ dr = \Gamma \left(1+\frac 2 p \right) \pi \gamma_p(s) \ .
\end{align*}
Then with $d_{p,n} = c_{p,n} \left(\Gamma \left(1+\frac 2 p \right) \pi \right)^n$,
$$\vol_{2(n-1)} \left( a^\perp \cap B_p^n(\C) \right) = d_{p,n} \int_0^\infty \prod_{j=1}^n \gamma_p(|a_j| s) \ s \ ds \ , $$
which yields for $a = a^{(1)}$ that $\vol_{2(n-1)} \left( B_p^{n-1}(\C) \right) = d_{p,n} \int_0^\infty \gamma_p(s) \ s \ ds$. For $n=2$, this gives
$\int_0^\infty \gamma_p(s) \ s \ ds = \frac{\vol_2( B_p^1(\C) )}{d_{p,2}} = \frac \pi {d_{p,2}} = \frac 2 {\Gamma \left(1+\frac 2 p \right)}$. Normalization yields
$$A_{n,p}(a) = \frac {\int_0^\infty \prod_{j=1}^n \gamma_p(|a_j| s) \ s \ ds} {\int_0^\infty \gamma_p(s) \ s \ ds} = \frac 1 2 \Gamma \left(1+\frac 2 p \right) \int_0^\infty \prod_{j=1}^n \gamma_p(|a_j| s) \ s \ ds \ .$$

\begin{corollary}\label{cor1}
$$A_{n,p}\left( a^{(2)} \right) = 2^{1-\frac 2 p}. $$
\end{corollary}

\begin{proof}
We note that $\gamma_p$ is -up to a constant- the Hankel transform of $\HH(f)$ of $f$, $f(r) = \exp(-r^p)$, $\HH(f)(s) = \int_0^\infty J_0(sr) f(r) \ r \ dr$. The Hankel transform is an isometry $\HH : L_2((0,\infty); r dr) \to L_2((0,\infty); r dr)$ with $\HH^2 = \Id$, cf. Poularikas \cite {P}, chapter 9. Therefore
\begin{align*}
A_{n,p}\left( a^{(2)} \right) &= \Gamma\left( 1+\frac 2 p \right) \frac 1 2 \int_0^\infty \gamma_p\left( \frac s {\sqrt 2} \right)^2 \ s \ ds \\
&= \Gamma\left( 1+\frac 2 p \right) \int_0^\infty \gamma_p(r)^2 \ r \ dr \\
&= \Gamma\left( 1+\frac 2 p \right) \left( \frac 2 {\Gamma\left( 1+\frac 2 p \right)} \right)^2  \int_0^\infty \exp(-2 r^p) \ r \ dr = 2^{1-\frac 2 p},
\end{align*}
since $\int_0^\infty \exp(-2 r^p) \ r \ dr = 2^{-\frac 2 p -1} \Gamma\left( 1+\frac 2 p \right)$.
\end{proof}

To prove Theorem \ref{th1}, we need some facts on the $\Gamma$-function.

\begin{lemma}\label{lem1}
(a) Let $f(p):= \frac{\Gamma\left( 1+\frac 4 p \right)}{\Gamma\left( 1+\frac 2 p \right)}$. Then $f(p) \ge \frac{24}{25}$ for all $p \ge 4$. \\
(b) Let $g(p):= \frac{\Gamma\left( 1+\frac 1 p \right)}{\Gamma\left( 1+\frac 2 p \right)}$. Then $g(p)$ is decreasing for all $p \ge 7$, with $g(7)< 1.0397$ and $g(9) < 1.0377$. \\
(c) Let $h(p):= \frac{\left( 2^{\frac 1 p} \Gamma\left( 1+\frac 2 p \right) \right)^2}{\Gamma\left( 1+\frac 4 p \right)}$. Then $h(p) > 1$ for all $2 < p < \infty$ and for all $p \ge 9$
$$h(p) \ge 1 + \frac{2 \ln 2} p -\frac{\frac 2 3 \pi^2 - 2 (\ln 2)^2} {p^2} + \frac 4 {p^3} > 1. $$
\end{lemma}

\begin{proof}
(a) Let $\Psi:= (\ln \Gamma)'$ denote the Digamma-function. Then $\Psi' >0$, since $\Gamma$ is logarithmic convex. We have
$$f'(p) = \frac{2 f(p)}{p^2} \left( \Psi\left( 1+\frac 2 p \right) - 2 \Psi\left( 1+\frac 4 p \right) \right). $$
The derivative of $F(p) := \Psi\left( 1+\frac 2 p \right) - 2 \Psi\left( 1+\frac 4 p \right)$ is
$$F'(p) = \frac 2 {p^2} \left( 4 \Psi'\left( 1+\frac 4 p \right) - \Psi'\left( 1+\frac 2 p \right) \right).$$
By Artin \cite{A} or Abramowitz, Stegun \cite{AS}, 6.3.16 and 6.4.10, we have for all $x > 0
$ that
\begin{equation}\label{eq2.3}
\Psi(1+x) = -\gamma + \sum_{n=1}^\infty \frac x {n(n+x)} \; , \; \Psi'(1+x) = \sum_{n=1}^\infty \frac 1 {(n+x)^2},
\end{equation}
where $\gamma \simeq 0.5772$ denotes the Euler constant. Therefore $\Psi'(1) = \zeta(2) = \frac{\pi^2} 6$, $\Psi'(1+x)$ is decreasing in $x$ and we have for $0 \le x \le 1$ that $\frac{\pi^2} 6 - 1 = \Psi'(2) \le \Psi'(1+x) \le \Psi'(1) = \frac{\pi^2} 6$. Here and later $\zeta$ denotes the Riemann $\zeta$-function,
$\zeta(\alpha) = \sum_{n=1}^\infty \frac 1 {n^\alpha}$ for $\alpha>1$. Hence for $p \ge 4$, $F'(p) \ge \frac 2 {p^2} \left( \frac{\pi^2} 2 - 4 \right) >0$ and $F$ is increasing. Further $F(13) \simeq -0.028$, $F(14) \simeq 0.071$: $F$ has exactly one zero $p_1 \in [4,\infty)$, $p_1 \simeq 13.78$. Thus $f$ is decreasing in $[4,p_1)$ and increasing in $(p_1,\infty)$. Hence for $p \ge 4$, $f(p) \ge f(p_1) > 0.9618 > \frac{24}{25}$. \\

(b) For $g$ we have
$$g'(p) = \frac{g(p)}{p^2} \left( 2 \Psi\left( 1+\frac 2 p \right) - \Psi\left( 1+\frac 1 p \right) \right). $$
Then $G(p):= 2 \Psi\left( 1+\frac 2 p \right) - \Psi\left( 1+\frac 1 p \right)$ satisfies
$$G'(p) = \frac 1 {p^2} \left( \Psi'\left( 1+\frac 1 p \right) - 4 \Psi'\left( 1+\frac 2 p \right) \right).$$
For $x \in [0,1]$ we have as in (a) $\frac{\pi^2} 6 -1 \le \Psi'(1+x) \le \frac {\pi^2} 6$. We find for all $p \ge 2$ that $G'(p) \le -\frac 1 {p^2} \left( \frac{\pi^2} 2 - 4 \right) < 0$, so that $G$ is decreasing. Since $G(7) < -0.007$, we have $G(p) < 0$ for all $p \ge 7$. Therefore $g$ is decreasing for $p \ge 7$, with $g(7) < 1.0390$ and $g(9) < 1.0377$. \\

(c) Let $H(p):= \frac{\Gamma\left( 1+\frac 2 p \right)^2}{\Gamma\left( 1+\frac 4 p \right)}$. We claim that for all $p \ge 9$
$$H(p) > 1 -\frac 2 3 \frac{\pi^2} p + \frac{15}{p^2}. $$
We have $H'(p) = H(p) \frac 4 {p^2} \left( \Psi\left( 1+\frac 4 p \right) - \Psi\left( 1+\frac 2 p \right) \right)$. Using equation \eqref{eq2.3} and the geometric series, we find for all $p>4$
\begin{align*}
(\ln H)'(p) = \frac{H'(p)}{H(p)} &= \frac 4 {p^2} \sum_{n=1}^\infty \left( \frac{\frac 4 p}{n\left(n+\frac 4 p\right)} -  \frac{\frac 2 p}{n\left(n+\frac 2 p\right)} \right) \\
&= \frac 4 {p^2} \sum_{n=1}^\infty \left( \sum_{k=0}^\infty \frac{(-1)^k}{n^{k+2}} \frac{4^{k+1}-2^{k+1}}{p^{k+1}} \right) \\
&= \frac 4 {p^2} \sum_{k=0}^\infty (-1)^k \zeta(k+2) \frac{4^{k+1}-2^{k+1}}{p^{k+1}}.
\end{align*}
This is an alternating series with decreasing coefficients $(\zeta(k+2) \frac{4^{k+1}-2^{k+1}}{p^{k+1}})_{k=0}^\infty$, using that $p>4$. Integration yields
$$(\ln H)(p) = C + 4 \sum_{k=0}^\infty (-1)^{k+1} \frac{\zeta(k+2)}{k+2} \frac{4^{k+1}-2^{k+1}}{p^{k+2}} = C - \frac 2 3 \frac{\pi^2}{p^2} + 16 \frac{\zeta(3)}{p^3} ... $$
Since $\Gamma(1+x) = 1 - \gamma x + O(x^2)$, $H(p) = 1 \pm O\left( \frac 1 {p^2} \right)$ and $\lim_{p \to \infty} (\ln H)(p) = 0$, so that the constant is zero, $C=0$. Since the series for $(\ln H)$ is alternating with decreasing coefficients for $p > 4$, we get a lower bound by truncating the series after three terms. Using $\zeta(2) = \frac{\pi^2} 6$, $\zeta(4) = \frac{\pi^4}{90}$ we find that
$$(\ln H)(p) \ge - \frac 2 3 \frac{\pi^2}{p^2} + \frac{16 \zeta(3)}{p^3}- \frac{28}{45} \frac{\pi^4}{p^4} =: \phi(p). $$
This implies by the series expansion of the exponential for $p \ge 9$ that
\begin{align*}
H(p) &= \exp((\ln H)(p)) \ge 1 + \phi(p) + \frac 1 2 \phi(p)^2 \\
&\ge 1 - \frac 2 3 \frac{\pi^2}{p^2} + \frac{16 \zeta(3)}{p^3} - \frac{\frac 2 5 \pi^4}{p^4} - \frac{32 \pi^2 \zeta(3)} {3 p^5} > 1 - \frac 2 3 \frac{\pi^2}{p^2} + \frac{13.3}{p^3} \ .
\end{align*}
The last inequality holds since $16 \zeta(3) -\frac{\frac 2 5 \pi^4} 9 - \frac {32 \pi^2 \zeta(3)} {3 \cdot 81} > 13.3$.
Further, $2^{\frac 2 p} = \exp\left( \frac{2 \ln 2} p \right) \ge 1 + \frac{2 \ln 2} p + \frac{2 (\ln 2)^2}{p^2}$, so that
\begin{align*}
h(p) = 2^{\frac 2 p} H(p) & > \left( 1 + \frac{2 \ln 2} p + \frac{2 (\ln 2)^2}{p^2} \right) \left( 1 - \frac 2 3 \frac{\pi^2}{p^2} + \frac{13.3}{p^3} \right) \\
& > 1 + \frac{2 \ln 2} p - \frac{\frac 2 3 \pi^2 - 2 (\ln 2)^2}{p^2} + \frac {4} {p^3} \ ,
\end{align*}
where the last inequality is true since the product expansion yields positive coefficients of $p^{-4}$, $p^{-5}$ and $p^{-6}$ and $13.3 - \frac 4 3 \pi^2 \ln(2) > 4$ holds for the coefficient of $p^{-3}$. As easily seen, the last expression is $>1$ for $p >9$ (even for $p>4$). \\

We have $h(2)=1$. To prove $h(p) > 1$ also for $2 < p < 9$, it suffices to show $(\ln h)(p) > 0$, i.e. $(\ln h)(p) = \frac {2 \ln 2} p + (\ln H)(p) > 0$, which is satisfied for $p > 4$ if
$$\frac {2 \ln 2} p - \frac{\frac 2 3 \pi^2}{p^2} + \frac{16 \zeta(3)}{p^3}- \frac{28}{45} \frac{\pi^4}{p^4} > 0. $$
This holds for all $p \ge 4.01$. Taking two more terms in the expansion for $(\ln H)(p)$ yields that $p \ge 3.82$ suffices. \\
For $2 < p < 4$ we check the sign of the derivative \\ $h'(p) = h(p) \left( -\frac{2 \ln 2}{p^2} - \frac 4 {p^2} \Psi\left( 1+\frac 2 p \right) + \frac 4 {p^2} \Psi\left( 1+\frac 4 p \right) \right)$. This is positive if and only if $\Psi\left( 1+\frac 4 p \right) - \Psi\left( 1+\frac 2 p \right) > \frac 1 2 \ln 2 \simeq 0.3466$.  Let
$K(p):= \Psi \left( 1+\frac 4 p \right) - \Psi \left( 1+\frac 2 p \right)$. Then $K'(p) = \frac 2 {p^2} \left( \Psi'\left(1+\frac 2 p \right) - 2 \Psi' \left(1+\frac 4 p \right) \right)$. By \eqref{eq2.3} $\Psi'(1+x)$ is decreasing in $x$. Hence for $2 \le p \le 4$, $\Psi'\left(1+\frac 2 p \right) - 2 \Psi' \left(1+\frac 4 p \right) \le \Psi'(2) - 2 \Psi'(3) = \frac 3 2 - \frac{\pi^2} 6  < - \frac 1 7 < 0$, since $\Psi'(2) = \frac{\pi^2} 6 -1$ and $\Psi'(3) = \frac{\pi^2} 6 - \frac 5 4$, see Abramowitz, Stegun \cite{AS}, 6.4.3. Hence $K$ is decreasing in $2 \le p \le 4$ and $K(p) \ge K(4) = \Psi(2) - \Psi \left(\frac 3 2 \right) = 2 \ln 2 -1 \simeq 0.3863 > \frac 1 2 \ln 2$ for all $p \in [2,4]$.
Therefore $h$ is strictly increasing in $[2,4]$ so that $h(p) > h(2) = 1$ for $2 < p \le 4$.
\end{proof}

To prove Theorem \ref{th2} we need some Lipschitz property of $A_{n,p}(a)$ with respect to $p$, similar to Lemma 14 in Eskenazis, Nayar, Tkocz \cite{ENT}.

\begin{proposition}\label{p-lipschitz}
    For $p > 8$ and every unit vector $a \in \C^n$ we have
    \begin{align*}
        |A_{n, p}(a) - A_{n, \infty}(a)| < \frac{16}{p}.
    \end{align*}
\end{proposition}

\begin{proof}
    Let $z$, $w$ be non-zero vectors in $\C^n$. It is proved in Koldobsky, Paouris, Zymonopoulou \cite{KPZ} that the function $N(z) := \frac{|z|}{\left( \vol_{2(n-1)}\left( \tilde{B}_\infty^n(\C) \cap z^{\perp} \right) \right)^{\frac 1 2}}$, where $\tilde{B}_\infty^n(\C) = \frac{1}{\pi}\D^n$ is a normalized polydisc, defines a norm on $\C^n$. Using this in the first inequality and writing for simplicity $\vol(x) := \vol_{2(n-1)}\left( \tilde{B}_\infty^n(\C) \cap x^{\perp} \right)$ we have
    \begin{align*}
        \left| N(z)^{-2} - N(w)^{-2} \right| &= \frac{\left| N(w)^2 - N(z)^2 \right|}{N(z)^2N(w)^2} = \frac{N(w) + N(z)}{N(z)^2N(w)^2}|N(w) - N(z)| \\ &\leq \frac{N(w) + N(z)}{N(z)^2N(w)^2}N(w - z) \\ &= \frac{\left( \frac{|w|}{\vol(w)^{1/2}} + \frac{|z|}{\vol(z)^{1/2}} \right)\vol(w)\vol(z)}{|w|^2|z|^2}\cdot\frac{|w - z|}{\vol(w - z)^{1/2}} \\ &\leq 4|w - z|\frac{|w| + |z|}{|w|^2|z|^2}
    \end{align*}
    since due to the Theorem of \cite{OP} we have $1 \leq \vol(\cdot) \leq 2$. Note that $N(a)^{-2} = A_{n, \infty}(a)$. By Proposition \ref{prop1} (and following the notation thereof) we also have
    \begin{align*}
        \frac{A_{n, p}(a)}{\Gamma\left( 1 + \frac 2 p \right)} = \E_R\E_{\xi}\left|\sum_{j=1}^na_jR_j\xi_j\right|^{-2} = \E_RN(aR)^{-2},
    \end{align*}
    where $aR = (a_1R_1, \ldots, a_nR_n)$ and the latter equality follows by the Remark following Proposition \ref{prop1}. Hence we have
    \begin{align}\label{i1i2}
        \left| \frac{A_{n, p}(a)}{\Gamma\left( 1 + \frac 2 p \right)} - A_{n, \infty}(a) \right| &= \left| \E N(aR)^{-2} - N(a)^{-2} \right| \leq 4\E\left[ |a - aR|\frac{|a| + |aR|}{|a|^2|aR|^2} \right] \notag \\ &= 4\E|a - aR||aR|^{-1} + 4\E|a - aR||aR|^{-2} = I_1 + I_2.
        \end{align}
    Using Cauchy-Schwarz inequality we obtain
    \begin{equation}\label{i1}
        I_1 \leq 4\sqrt{\E|a-aR|^2}\sqrt{\E|aR|^{-2}} = 4\sqrt{\E\sum_{j=1}^na_j^2(R_j-1)^2}\sqrt{\E\left( \sum_{j=1}^na_j^2R_j^2 \right)^{-1}}
    \end{equation}
    and
    \begin{equation}\label{i2}
        I_2 \leq 4\sqrt{\E\sum_{j=1}^na_j^2(R_j-1)^2}\sqrt{\E\left( \sum_{j=1}^na_j^2R_j^2 \right)^{-2}}.
    \end{equation}
    By convexity of $x \mapsto \frac{1}{x}$ and $x \mapsto \frac{1}{x^2}$ for $x > 0$ and Jensen's inequality (recall that $\sum_{j=1}^na_j^2 = 1$) we have
    \begin{align}\label{i12}
        \E\left(\sum_{j=1}^na_j^2R_j^2\right)^{-1} \leq \E\sum_{j=1}^na_j^2R_j^{-2} = \frac{1}{\Gamma\left( 1 + \frac 2 p \right)}
    \end{align}
    and
    \begin{align}\label{i22}
        \E\left(\sum_{j=1}^na_j^2R_j^2\right)^{-2} \leq \E\sum_{j=1}^na_j^2R_j^{-4} = \frac{\Gamma\left( 1 - \frac 2 p \right)}{\Gamma\left( 1 + \frac 2 p \right)}.
    \end{align}
    To bound $\E\sum_{j=1}^na_j^2(R_j - 1)^2 = \E(R_1 - 1)^2 = \frac{\Gamma\left( 1 + \frac 4 p \right) - 2\Gamma\left( 1 + \frac 3 p \right) + \Gamma\left( 1 + \frac 2 p \right)}{\Gamma\left( 1 + \frac 2 p \right)}$, we consider the function $h(x) := \Gamma(1+4x) - 2\Gamma(1+3x) + \Gamma(1+2x)$. We have $h(0) = 0$ and $h'(0) = 0$, hence for small $x > 0$ there exists $0 < \theta < x$ such that $h(x) = \frac{1}{2}x^2h''(\theta)$. As $\Gamma''(1) < 2$ and $\Gamma''$ is decreasing on $(1, 3/2)$, by computing $h''(\theta)$ we obtain $h(x) \leq 2x^2$ for $x < 1/8$. Hence
    \begin{align}\label{i11i21}
        \E\sum_{j=1}^na_j^2(R_j - 1)^2 = \frac{h\left( \frac 1 p \right)}{\Gamma\left( 1 + \frac 2 p \right)} \leq \frac{2}{p^2\Gamma\left( 1 + \frac 2 p \right)}
    \end{align}
    for all $p > 8$. Putting \eqref{i12}, \eqref{i11i21} into \eqref{i1}, \eqref{i22}, \eqref{i11i21} into \eqref{i2} and \eqref{i1}, \eqref{i2} into \eqref{i1i2} we get
    \begin{align*}
        \left| A_{n, p}(a) - A_{n,\infty}(a) \right| &\leq \left| A_{n, p}(a) - \Gamma\left( 1 + \frac 2 p \right)A_{n,\infty}(a) \right| + A_{n,\infty}(a)\left| \Gamma\left( 1 + \frac 2 p \right) - 1 \right| \\ &\leq 4\Gamma\left( 1 + \frac 2 p \right)\sqrt{\frac{2}{p^2\Gamma\left( 1 + \frac 2 p \right)}} \cdot \frac{1 + \sqrt{\Gamma\left( 1 - \frac 2 p \right)}}{\sqrt{\Gamma\left( 1 + \frac 2 p \right)}} \\ &+ 2\left( 1 - \Gamma\left( 1 + \frac 2 p \right) \right) \leq \frac{1}{p}\left(4\sqrt{2}\left(1 + \sqrt[4]{\pi}\right) + 4\gamma\right) < \frac{16}{p},
    \end{align*}
    where in the second last inequality we used $\Gamma(1+x) > 1 - \gamma x$ for $x > 0$, which follows by $\Gamma(1) = 1$, $\Gamma'(1) = -\gamma$ and convexity of $\Gamma$, and $\Gamma(1 - x) < \Gamma\left( \frac 1 2 \right) = \sqrt{\pi}$ for $0 < x < \frac 1 2$.
\end{proof}

\section{Proof of Theorem \ref{th1}}

   We start with proof that $\lim_{n \to \infty}A_{n, p}\left(a^{(n)}\right) = 2\frac{\Gamma\left(1 + \frac 2 p\right)^2}{\Gamma\left(1 + \frac 4 p\right)}$ and that it is greater than $A_{n, p}\left(a^{(2)}\right)$. Define $X_n = \frac{1}{\sqrt{n}}\sum_{j=1}^nR_j\xi_j$. Then, according to the central limit theorem, $X_n$ converges in distribution to $\sigma G$, where $G$ is a standard Gaussian vector in $\R^4$ with mean 0 and the identity covariance matrix, and $\sigma^2 = \frac{1}{4}\E R_1^2 = \frac{1}{4}\frac{\Gamma\left( 1 + \frac 4 p \right)}{\Gamma\left( 1 + \frac 2 p \right)}$. Note that for $p \ge 4$ we have $\sigma^2 \ge \frac 1 4 \cdot \frac{24}{25} = \frac{6}{25}$ by Lemma \ref{lem1} (a) and for $p \in (2, 4)$ we have $\sigma^2 \ge \frac 1 4$ since then $\Gamma\left(1 + \frac 4 p\right) > 1 > \Gamma\left(1 + \frac 2 p\right)$ due to $1 + \frac 4 p > 2 > 1 + \frac 2 p$. Our aim is to show the convergence of the second negative moments. Using the fact that $|G|^2$ has density $\frac{x}{4}e^{-\frac{x}{2}}\mathbf{1}_{x>0}$ we get
    \begin{align*}
        \E|\sigma G|^{-2} = \frac{1}{\sigma^2}\E|G|^{-2} = \frac{1}{\sigma^2}\int_0^{\infty}\frac{1}{x}\frac{x}{4}e^{-\frac x 2}dx = \frac{1}{2\sigma^2} = 2\frac{\Gamma\left( 1 + \frac 2 p \right)}{\Gamma\left( 1 + \frac 4 p \right)} =: C_p.
    \end{align*}
    To verify the convergence of the second negative moments, denote \\
    $X_n^N := \frac{1}{\sqrt{N-n}}\sum_{j=n+1}^NR_j\xi_j$ for $n < N$ and $\delta_N := \left|\E|\sigma G|^{-2} - \E|X_N|^{-2}\right|$. We shall prove by induction on $N$ that $\delta_N \le \frac{C_0}{\sqrt{N}}$ with some universal constant $C_0$ to be chosen later. The vectors $X_n$ and $X_n^N$ are independent and we have
    \begin{align*}
        X_N = \frac{\sqrt{n}}{\sqrt{N}}X_n + \frac{\sqrt{N-n}}{\sqrt{N}}X_n^N.
    \end{align*}
    Let $G_1$, $G_2$ be independent standard Gaussian vectors such that
    \begin{align*}
        \sigma G = \frac{\sqrt{n}}{\sqrt{N}}\sigma G_1 + \frac{\sqrt{N-n}}{\sqrt{N}}\sigma G_2.
    \end{align*}
    For $t > 0$ we define
    \begin{align*}
        f_1(t) &:= \P\left( \left| \frac{\sqrt{n}}{\sqrt{N}}X_n \right|^{-2} > t \right), &f_2(t) := \P\left( \left| \frac{\sqrt{N-n}}{\sqrt{N}}X_n^N \right|^{-2} > t \right), \\ g_1(t) &:= \P\left( \left| \frac{\sqrt{n}}{\sqrt{N}}\sigma G_1 \right|^{-2} > t \right), &g_2(t) := \P\left( \left| \frac{\sqrt{N-n}}{\sqrt{N}}\sigma G_2 \right|^{-2} > t \right).
    \end{align*}
    Using Lemma 2 in Glover, Tkocz, Wyczesany \cite{GTW} and writing the expectation in terms of cumulative distribution function as usual we get
    \begin{align}\label{cutinhalf}
        \delta_N &= \left| \int_0^{\infty}g_1(t)g_2(t)dt - \int_0^{\infty}f_1(t)f_2(t)dt \right| \notag \\ &\leq \int_0^{\infty}g_1(t)|g_2(t) - f_2(t)|dt + \int_0^{\infty}f_2(t)|g_1(t) - f_1(t)|dt.
    \end{align}
    We will bound these integrals using the fact that integrals of $f_i$, $g_i$ are bounded and then apply a Berry-Esseen type bound to $|f_i - g_i|$ pointwise. To optimize this method we choose $n = \left\lfloor \frac{N}{2} \right\rfloor$ (here we assume $N > 1$). Using Theorem 1.1 in Rai\v c \cite{R} with summands $\frac{R_i\xi_i}{\sigma\sqrt{n}}$, $i = 1, \ldots, n$ and $A = \left\{x \in \R^4 : |x| < \frac{\sqrt{N}}{\sigma\sqrt{nt}}\right\}$ and denoting $C = 42\sqrt{2} + 16$ we have
    \begin{align}\label{g1f1}
        |g_1(t) - f_1(t)| &= \left| \P\left( |G_1| < \frac{\sqrt{N}}{\sigma\sqrt{nt}} \right) - \P\left( \left| \frac{X_n}{\sigma} \right| < \frac{\sqrt{N}}{\sigma\sqrt{nt}} \right) \right| \notag \\ &\leq C\sum_{j=1}^n\E\left| \frac{R_j\xi_j}{\sigma\sqrt{n}} \right|^3 = \frac{C}{\sigma^3\sqrt{n}}\E R_1^3 \le \frac{C\sqrt{3}}{\sigma^3\sqrt{N}}\E R_1^3 =: \frac{C_1}{\sqrt{N}},
    \end{align}
    and similarly
    \begin{equation}\label{g2f2}
        |g_2(t) - f_2(t)| \le \frac{C_1}{\sqrt{N}}.
    \end{equation}
    We know that $\E|G|^{-2} = \frac{1}{2}$. Moreover, by the triangle inequality and induction hypothesis we have $\E|X_m|^{-2} \le \E|\sigma G|^{-2} + \delta_m \le C_p + \frac{C_0}{\sqrt{m}}$ for $m \in \N$, $m < N$. Hence
    \begin{equation}\label{g1}
        \int_0^{\infty}g_1(t)dt = \E\left| \frac{\sqrt{n}}{\sqrt{N}}\sigma G_1 \right|^{-2} = \frac{N}{2n\sigma^2} \leq \frac{25N}{12n} \le \frac{75}{12}
    \end{equation}
    and
    \begin{equation}\label{f2}
        \int_0^{\infty}f_2(t)dt = \E\left| \frac{\sqrt{N-n}}{\sqrt{N}}X_n^N \right|^{-2} \le 2\left(C_p + \frac{C_0}{\sqrt{N-n}}\right).
    \end{equation}
    Combining \eqref{cutinhalf}, \eqref{g1f1}, \eqref{g2f2}, \eqref{g1} and \eqref{f2} we get
    \begin{align*}
        \delta_N \le \frac{C_1}{\sqrt{N}}\left(\frac{75}{12} + 2C_p + \frac{2C_0}{\sqrt{N-n}}\right) \le \frac{C_0}{\sqrt{N}},
    \end{align*}
    provided that $C_1\left(\frac{75}{12} + 2C_p + \frac{2C_0}{\sqrt{N-n}}\right) \le C_0$. The latter is true for sufficiently large $C_0$ if only $\frac{C_1}{\sqrt{N-n}} < \frac 1 2$. Take $N_0$ such that $\frac{C_1}{\sqrt{N-n}} < 0.49$ for $N > N_0$ and choose $C_0 = \max\left\{50C_1\left(\frac{75}{12} + 2C_p\right), \delta_1, \frac{\delta_2}{\sqrt{2}}, \ldots, \frac{\delta_{N_0}}{\sqrt{N_0}}\right\}$, then the induction with trivial basis $N \le N_0$ and the inductive step being the preceding proof for $N > N_0$ proves that $\delta_N \le \frac{C_0}{\sqrt{N}}$. Hence $\delta_N \to 0$ with $N \to \infty$, which proves the convergence of the second negative moments. Therefore
    \begin{align*}
        \lim_{n \to \infty}A_{n, p}\left( a^{(n)} \right) = \Gamma\left(1 + \frac 2 p\right) C_p > A_{n, p}\left( a^{(2)} \right),
    \end{align*}
    where the inequality follows from Corollary \ref{cor1} and Lemma \ref{lem1} (c).\\

\vspace{0,5cm}

We already proved the inequality $\lim_{n \to \infty}A_{n, p}\left( a^{(n)} \right) > A_{n, p}\left( a^{(2)} \right)$ and the existence of $N(p)$. What remains to be proved are the estimates for $N(p)$. By Proposition \ref{prop2}
$$A_{n,p}\left( a^{(n)} \right) = \Gamma\left( 1+\frac 2 p \right) \frac 1 2 \int_0^\infty \left( \gamma_p\left( \frac s {\sqrt n} \right) \right)^n \ s \ ds    \ , $$
where
\begin{equation}\label{eq2.4}
\gamma_p\left( \frac s {\sqrt n} \right) = \frac 2 {\Gamma\left( 1+\frac 2 p \right)} \int_0^\infty J_0\left( \frac{sr}{\sqrt n} \right) \exp(-r^p) \ r \ dr.
\end{equation}
To find finite values $n$ for which $A_{n, p}\left( a^{(n)} \right) > A_{n, p}\left( a^{(2)} \right)$ holds, we estimate this from  below. We use that for $0 \le x \le 2$, $J_0(x) >0$ and $J_0(x) \ge 1 - \frac{x^2} 4 + \frac{x^4}{72}$, since by the series representation of $J_0$ with $\frac 1 {64} = \frac 1 {72} + \frac 1 {576}$
$$J_0(x) - \left( 1 - \frac{x^2} 4 + \frac{x^4}{72} \right) = \frac{x^4}{576} - \frac{x^6}{2304} + \sum_{m=4}^\infty \frac{(-1)^m}{m!^2} \left( \frac x 2 \right)^{2m} > 0$$
for $0 < x \le 2$. We note that the first zero of $J_0$ is at $x_1 \simeq 2.4048 > 2$. This implies that
$$I_1 := \int_0^{\frac {2 \sqrt n} s} J_0\left( \frac{sr}{\sqrt n} \right) \exp(-r^p) \ r \ dr \ge \int_0^{\frac {2 \sqrt n} s} \left( r-\frac{s^2}{4n} r^3+\frac{s^4}{72 n^2} r^5 \right) \exp(-r^p) dr. $$
Writing $\int_0^{\frac {2 \sqrt n} s} = \int_0^\infty - \int_{\frac {2 \sqrt n} s}^\infty$, the integral over $(0,\infty)$ can be evaluated in terms of Gamma-functions and the remainder will be estimated. We find that
\begin{align*}
    I_1 \ge \frac 1 2 \Gamma\left( 1+\frac 2 p \right) - \frac{s^2}{4n} \frac 1 4 \Gamma\left( 1+\frac 4 p \right) + \frac{s^4}{72 n^2} \frac 1 6 \Gamma\left( 1+\frac 6 p \right) - Q_1,
\end{align*}
where for $p \ge 6$ and $s \le 2 \sqrt n$
\begin{align*}
Q_1 &:= \int_{\frac {2 \sqrt n} s}^\infty \left( r-\frac{s^2}{4n} r^3+\frac{s^4}{72 n^2} r^5 \right) \exp(-r^5) dr \\
& = \frac 1 p \int_{\left( \frac {2 \sqrt n} s \right)^p}^\infty \left( u^{\frac 2 p -1} - \frac{s^2}{4n} u^{\frac 4 p -1} + \frac{s^4}{72 n^2} u^{\frac 6 p -1} \right) \exp(-u) du \\
& \le \frac 1 p \int_{\left( \frac {2 \sqrt n} s \right)^p}^\infty \left( 1+\frac{s^4}{72 n^2} \right) \exp(-u) du = \frac 1 p \left( 1+\frac{s^4}{72 n^2} \right) \exp\left( -\left( \frac {2 \sqrt n} s \right)^p \right).
\end{align*}
By Gradshteyn, Ryszik \cite[8.479]{GR}, we have that $|J_0(x)| \le \sqrt{\frac 2 \pi} \frac 1 {\sqrt x} < \frac 4 5 \frac 1 {\sqrt x}$ for any $x>0$. This implies that
\begin{align*}
Q_2 & := \int_{\frac {2 \sqrt n} s}^\infty J_0\left( \frac{sr}{\sqrt n} \right) \exp(-r^p) \ r \ dr \le |Q_2|
\le \frac 4 5 \int_{\frac {2 \sqrt n} s}^\infty \frac{n^{\frac 1 4}}{\sqrt s} r^{\frac 1 2} \exp(-r^p) dr \\
& = \frac 4 5 \frac 1 p \frac{n^{\frac 1 4}}{\sqrt s} \int_{\left( \frac {2 \sqrt n} s \right)^p}^\infty u^{\frac 3 {2p} - 1} \exp(-u) du \le \frac 4 5 \frac 1 p \frac{n^{\frac 1 4}}{\sqrt s} \exp\left( -\left( \frac {2 \sqrt n} s \right)^p \right).
\end{align*}
For $x \ge 1$, $\exp(-x) \le \frac 1 { e x }$. Thus for $p \ge 6$ and $s \le 2 \sqrt n$,
\begin{align*}
Q_1+|Q_2| & \le \frac 1 p \left( 1+\frac{s^4}{72 n^2} + \frac 4 5 \frac{n^{\frac 1 4}}{\sqrt s} \right) \exp\left( -\left( \frac {2 \sqrt n} s \right)^p \right) \\
& \le \frac 1 {e p} \left( 1+\frac{s^4}{72 n^2} + \frac 4 5 \frac{n^{\frac 1 4}}{\sqrt s} \right) \left( \frac s {2 \sqrt n} \right)^p.
\end{align*}
We now further restrict ourselves to $s \le \sqrt{2 n}$ and claim that the right side then is $< \frac{s^4}{500 n^2}$ for all $p \ge 8$. This requires
$$\frac{ s^{p-4} \left( 1+\frac{s^4}{72 n^2} + \frac 4 5 \frac{n^{\frac 1 4}}{\sqrt s} \right) } {\left(\sqrt{2 n}\right)^{p-4}} < \frac {e p \left(\sqrt 2\right)^p}{125}. $$
The left side is increasing in $s$ and maximal for $s= \sqrt{ 2 n }$ and then bounded by $\frac 7 4$. We thus want $ \frac 7 4 125 < e p \left(\sqrt 2\right)^p$ which is satisfied for all $p \ge 8$. Hence for $p \ge 8$ and $s \le \sqrt{2 n}$
\begin{align*}
\gamma_p\left( \frac s {\sqrt n} \right) & = \frac 2 {\Gamma\left( 1+\frac 2 p \right)} \left( \int_0^{\frac{2 \sqrt n}s} + \int_{\frac{2 \sqrt n}s}^\infty \right) \ J_0\left( \frac{sr}{\sqrt n} \right) \exp(-r^p) \ r \ dr \\
& \ge 1 - \frac 1 8 \frac{\Gamma\left( 1+\frac 4 p \right)}{\Gamma\left( 1+\frac 2 p \right)} \frac{s^2} n + \frac 1 {216} \frac{\Gamma\left( 1+\frac 6 p \right)}{\Gamma\left( 1+\frac 2 p \right)} \frac{s^4}{n^2} - \frac 2 {\Gamma\left( 1+\frac 2 p \right)} (Q_1 - Q_2) \\
& \ge 1 - \frac 1 8 \frac{\Gamma\left( 1+\frac 4 p \right)}{\Gamma\left( 1+\frac 2 p \right)} \frac{s^2} n + \frac 1 {216} \frac{\Gamma\left( 1+\frac 6 p \right)}{\Gamma\left( 1+\frac 2 p \right)} \frac{s^4}{n^2} - \frac 2 {\Gamma\left( 1+\frac 2 p \right)} (Q_1 + |Q_2|) \\
& \ge 1 - \frac 1 8 \frac{\Gamma\left( 1+\frac 4 p \right)}{\Gamma\left( 1+\frac 2 p \right)} \frac{s^2} n.
\end{align*}
The last inequality holds since $\Gamma(x) \ge 0.885 > \frac 7 8$ for all $1 \le x \le 2$ and $\frac 7 8 \frac 1 {216} > \frac 1 {250}$.
Let $c := \frac 1 8 \frac{\Gamma\left( 1+\frac 4 p \right)}{\Gamma\left( 1+\frac 2 p \right)}$. Then $c \le \frac 1 8$ for all $p \ge 9$, since $\Gamma$ is decreasing in $[1,1.46]$. Further by Lemma \ref{lem1} (a) $c \ge \frac 3 {25}$. For $0 \le x \le \frac 1 4$ we have that
$$\ln(1-x) = - \sum_{j=1}^\infty \frac{x^j} j \ge -x - \frac 1 2 x^2 \sum_{k=0}^\infty x^k = -x - \frac 1 2 \frac{x^2}{1-x} \ge -x-\frac 2 3 x^2 $$
and hence for $s \le \sqrt{2 n}$ with $x:=c \frac{s^2} n \le \frac 1 4$ and $\exp(-y) \ge 1 - y$,
\begin{align*}
    \left( 1-c \frac{s^2} n \right)^n &= \exp\left( n \ln\left( 1-c \frac{s^2} n \right) \right) \ge \exp\left( -c s^2 - \frac 2 3 c^2 \frac{s^4} n \right) \\ &\ge \exp\left( -c s^2 \right) \left( 1-\frac 2 3 c^2 \frac{s^4} n \right).
\end{align*}
Therefore, using $\int_0^{\sqrt {2n}} = \int_0^\infty - \int_{\sqrt {2n}}^\infty$ and
$$\int_0^\infty s \exp\left( -c s^2 \right) ds = \frac 1 {2 c}, \  \int_0^\infty s^5 \exp\left( -c s^2 \right) ds = \frac 1 {c^3}, $$
\begin{align}\label{eq2.5}
\int_0^{\sqrt {2n}} \gamma_p\left( \frac s {\sqrt n} \right)^n \ s \ ds & \ge \int_0^{\sqrt {2n}} \left( 1 - \frac 2 3 c^2 \frac{s^4} n \right) \exp\left( -c s^2 \right) \ s \ ds \nonumber \\
& = \frac 1 {2 c} \left( 1 - \frac 4 3 \frac 1 n \right) - \int_{\sqrt {2n}}^\infty \left( 1 - \frac 2 3 c^2 \frac{s^4} n \right) \exp\left( -c s^2 \right) \ s \ ds \nonumber \\
& \ge \frac 1 {2 c} \left( 1 - \frac 4 3 \frac 1 n \right) = 4 \frac{\Gamma\left( 1+\frac 2 p \right)}{\Gamma\left( 1+\frac 4 p \right)} \left( 1 - \frac 4 3 \frac 1 n \right)
\end{align}
for all $n \ge 16$ since
$$ -\int_{\sqrt {2n}}^\infty \left( 1 - \frac 2 3 c^2 \frac{s^4} n \right) \exp\left( -c s^2 \right) \ s \ ds = + \frac{\exp(-2 c n)}{6 c n} \left( 8 c^2 n^2 + 8 c n - 3 n + 4 \right) $$
is positive: the factor $8c^2n^2+8cn-3n+4$ increases with $c$ and is positive for $c= \frac 3 {25}$ and $n \ge 16$.  \\

By Gradshteyn, Ryszik \cite[8.479]{GR}, we have $|J_1(x)| \le \frac{\sqrt{\frac 2 \pi}}{\left(x^2-1\right)^{\frac 1 4}}$. This is $< \frac 1 2$ for all $x \ge 3$. The smallest positive zero of $J_1$ is $x_1 \simeq 3.812$. Thus for $x \in (0,3]$, $J_1(x) > 0$. The derivative $J_1'(x) = J_0(x) - \frac{J_1(x)} x$ has exactly one zero $x_0$ in $[0,3]$, $x_0 \simeq 1.8412$. Hence the absolute maximum of $|J_1(x)|$ for $x \ge 0$ satisfies $|J_1(x)| \le J_1(x_0) \le M:=0.5819$. By the proof of Proposition \ref{prop2}, \\ $\gamma_p(x) = \frac{2 c_p^{-1}} x \int_0^\infty J_1(xr) r^p \exp(-r^p) dr$ for $x>0$, where $c_p = \frac{\Gamma(1+\frac 2 p)} p$. Thus
$$|\gamma_p(x)| \le  \frac{2 c_p^{-1} M} x \int_0^\infty r^p \exp(-r^p) dr = \frac{2 M} x \frac{\Gamma\left( 1+\frac 1 p \right)}{\Gamma\left( 1+\frac 2 p \right)}. $$
By Lemma \ref{lem1} (b), $g(p) := \frac{\Gamma\left( 1+\frac 1 p \right)}{\Gamma\left( 1+\frac 2 p \right)}$ is decreasing for $p \ge 9$ and $g(p) \le g(9) < 1.0377$. We conclude that
$|\gamma_p(x)| \le \frac{2 M 1.0377} x < \frac{1.2077} x$ for all $x >0$ and $p \ge 9$. This implies the tail estimate
\begin{align}\label{eq2.6}
\int_{\sqrt{2n}}^\infty \left| \gamma_p\left( \frac s {\sqrt n} \right) \right|^n \ s \ ds & = n \int_{\sqrt 2}^\infty |\gamma_p(x)|^n \ x \ dx \le n \ 1.2077^n \int_{\sqrt 2}^\infty x^{-n + 1} dx \nonumber \\
& = \frac{2 n}{n-2} \left(\frac{1.2077}{\sqrt 2}\right)^n \le \frac{2 n}{n-2} 0.854^n.
\end{align}

We conclude from \eqref{eq2.4} and \eqref{eq2.6}, using $\Gamma\left( 1+\frac 4 p \right) \le \Gamma\left( 1+\frac 2 p \right)$ for $p \ge 9$ as well as \eqref{eq2.5} for $n \ge 16$ that
\begin{align*}
A_{n,p}\left( a^{(n)} \right) & = \Gamma\left( 1+\frac 2 p \right) \frac 1 2 \left( \int_0^{\sqrt{2n}} + \int_{\sqrt{2n}}^\infty \right) \ \gamma_p\left( \frac s {\sqrt n} \right)^n \ s \ ds \\
& \ge \Gamma\left( 1+\frac 2 p \right) \frac 1 2 \left( 4 \frac{\Gamma\left( 1+\frac 2 p \right)}{\Gamma\left( 1+\frac 4 p \right)} \left( 1 - \frac 4 3 \frac 1 n \right) - \frac {2 n} {n-2} 0.854^n \right) \\
& \ge 2 \frac{\Gamma\left( 1+\frac 2 p \right)^2}{\Gamma\left( 1+\frac 4 p \right)} \left(1 - \frac 4 3 \frac 1 n -  \frac 1 2 \frac n {n-2} 0.854^n \right).
\end{align*}
By Corollary \ref{cor1}, $A_{n,p}\left( a^{(2)} \right) = 2^{1-\frac 2 p}$. Therefore $A_{n,p}\left( a^{(n)} \right) > A_{n,p}\left( a^{(2)} \right)$ will hold provided that
$$F(p,n) := \frac{\left( 2^{\frac 1 p} \Gamma\left( 1+\frac 2 p \right) \right)^2}{\Gamma\left( 1+\frac 4 p \right)} \left(1 - \frac 4 3 \frac 1 n -  \frac 1 2 \frac n {n-2} 0.854^n \right) > 1 $$
is satisfied. By Lemma \ref{lem1} (c) a sufficient condition for this is that
$$G(p,n) := \left(1 + \frac {2 \ln 2} p - \frac{\frac 2 3 \pi^2 - 2 (\ln 2)^2}{p^2} + \frac 4 {p^3} \right) \left(1 - \frac 4 3 \frac 1 n -  \frac 1 2 \frac n {n-2} 0.854^n \right) > 1 $$
holds. For $p \ge 9$ and $n \ge \frac 5 2 p > 22$ we have that $\frac 4 3 \frac 1 n + \frac 1 2 \frac n {n-2} 0.854^n < \frac 8 {15} \frac 1 p + \frac{11}{20} 0.854^{\frac 5 2 p} < \frac 7 {10} \frac 1 p$. The last inequality is equivalent to $p \ 0.854^{\frac 5 2 p} < \frac{10}{33}$. Note that $p \ 0.854^{\frac 5 2 p}$ is decreasing in $p \ge 9$. The last inequality is correct for $p=9$ and hence for all $p \ge 9$. Thus $1 - \frac 4 3 \frac 1 n -  \frac 1 2 \frac n {n-2} 0.854^n > 1 - \frac{0.7} p$ for $p \ge 9$ and $n \ge \frac 5 2 p$. Further $2 \ln 2- \frac{\frac 2 3 \pi^2 - 2 (\ln 2)^2}p + \frac 4 {p^2} \ge 0.81$ for all $p \ge 9$. Hence for all $p \ge 9$ and $n \ge \frac 5 2 p$
$$G(p,n) \ge \left( 1+\frac{0.81} p \right)\left( 1-\frac{0.7} p \right) > 1. $$
The last inequality is equivalent to $\frac{0.11} p - \frac{0.81 \cdot 0.7}{p^2} > 0$ or $p > \frac{0.81 \cdot 0.7}{0.11} \simeq 5.15$, thus satisfied for $p \ge 9$. Similarly, we have for all $p \ge 140$ and $n \ge p$ that
$$G(p,n) \ge \left( 1+\frac{1.3464} p \right)\left( 1-\frac{1.3334} p \right) > 1. $$
This proves $A_{n,p}\left( a^{(n)} \right) > A_{n,p}\left( a^{(2)} \right)$ for $p \ge 9$ and $n \ge \frac 5 2 p$ as well as for $p \ge 140$ and $n \ge p$. \\

Similar to the lower estimate for $A_{n,p}\left( a^{(n)} \right)$, one may prove an upper estimate up to small error terms in $n$, leading to
$$\lim_{n \to \infty} A_{n,p}\left( a^{(n)} \right) = \frac{2 \ \Gamma\left( 1+\frac 2 p \right)^2}{\Gamma\left( 1+\frac 4 p \right)}, $$
which was also derived by the central limit theorem.       \hfill $\Box$

\section{Proof of Theorem \ref{th2}}

We may assume that $a_1 \geq a_2 \geq \ldots \geq a_n \geq 0$. Denote $c_1 = 1520$, $c_2 = 2 \cdot 10^{41}$ and $\delta(a) = \left| a - a^{(2)} \right|^2 = 2 - \sqrt{2}(a_1 + a_2)$. We shall follow closely the proof of \cite[Theorem 1]{ENT}, making necessary adjustments to the complex setting. As in that proof, we shall consider two cases: when $\delta(a)$ is large and when $\delta(a)$ is small.

\subsection{The vector $a$ is far from the extremizer.}

Suppose $\sqrt{\delta(a)} \geq \frac{c_2}{p}$. We have
\begin{align*}
    \sum_{j=1}^na_j^4 \geq \frac{\left( \sum_{j=1}^na_j^2 \right)^2}{n} = \frac{1}{n} \geq \frac{c_1}{p}.
\end{align*}
Thus, using Proposition \ref{p-lipschitz}, \cite[Theorem 1]{GTW} and $2^{1-x} \geq 2(1 - x\log 2) \geq 2 - 2x$ for $x > 0$, we get
\begin{align*}
    A_{n,p}(a) &\leq A_{n,\infty}(a) + |A_{n,p}(a) - A_{n,\infty}(a)| < \frac{16}{p} + 2 - \min\left\{ 10^{-40}\sqrt{\delta(a)}, \frac{1}{76n} \right\} \\ &\leq 2 + \frac{16}{p} - \frac{20}{p} \leq 2^{1 - \frac 2 p}.
\end{align*}

\subsection{The vector $a$ is close to the extremizer.}

Suppose $\sqrt{\delta(a)} < \frac{c_2}{p}$. Then $\frac{1}{\sqrt{2}} - \frac{c_2}{p} \leq a_2 \leq a_1 \leq \frac{1}{\sqrt{2}} + \frac{c_2}{p}$. Our aim is to show that
\begin{align*}
    \E\left| \sum_{j=1}^na_jR_j\xi_j \right|^{-2} \leq \E\left| \frac{R_1\xi_1 + R_2\xi_2}{\sqrt{2}} \right|^{-2} = \frac{2^{1-\frac 2 p}}{\Gamma\left( 1+\frac 2 p \right)} =: C'_p.
\end{align*}
We shall proceed by induction on $n$. The basic cases of $n = 2$ and $\sqrt{\delta(a)} \geq \frac{c_2}{p}$ are easy or already done. Note that $A_{2, p}(a) = \frac{1}{\|(a_1, a_2)\|_p^2}$, which can be computed in the same way as we computed $A_{2, p}\left( a^{(2)} \right)$ in Corollary \ref{cor1}. Let us pass to the inductive step. \\

Let $X = a_1R_1\xi_1 + a_2R_2\xi_2$, $Y = \sum_{j=3}^na_jR_j\xi_j$ and assume $Y \neq 0$ (otherwise the statement is trivial). Then $X$ and $Y$ are independent rotationally invariant random vectors in $\R^4$. By the inductive hypothesis we have
\begin{align*}
    \E|Y|^{-2} = \frac{1}{1 - a_1^2 - a_2^2}\E\left| \frac{Y}{\sqrt{1 - a_1^2 - a_2^2}} \right|^{-2} \leq \frac{C'_p}{1 - a_1^2 - a_2^2} =: \alpha^{-2}.
\end{align*}
Using this, \cite[Lemma 2]{GTW} and the concavity of $t \mapsto \min\left\{|X|^{-2}, t\right\}$, we find
\begin{align*}
    \E|X + Y|^{-2} \leq \E\min\left\{ |X|^{-2}, \alpha^{-2} \right\} = \E|X|^{-2} - \E\left( |X|^{-2} - \alpha^{-2} \right)_+.
\end{align*}
As $\E|X|^{-2} = \frac{1}{\|(a_1, a_2)\|_p^2\Gamma\left( 1+\frac 2 p \right)}$, the statement of Theorem \ref{th2} reduces to
\begin{align}\label{thmax}
    \E\left( |X|^{-2} - \alpha^{-2} \right)_+ \geq \frac{1}{\|(a_1, a_2)\|_p^2\Gamma\left( 1+\frac 2 p \right)} - C'_p = C'_p\left( \frac{2^{\frac 2 p-1}}{\|(a_1, a_2)\|_p^2} - 1 \right).
\end{align}

Assume $\|(a_1, a_2)\|_p < 2^{\frac 1 p - \frac 1 2}$, the other case is trivial. We begin with an upper estimate of the right hand side of \eqref{thmax}. For $p > 200$ we have $\frac 2 p < 100$ and thus, since $\Gamma(1 + x) \ge 1 - \gamma x$ for $x > -1$ by the convexity of $\Gamma$,
\begin{align*}
    C'_p = \frac{2^{1-\frac 2 p}}{\Gamma\left(1 + \frac 2 p\right)} \le \frac{2}{1 - \frac{2\gamma}{p}} \le \frac{2}{1 - \frac{\gamma}{100}} \le 2.03.
\end{align*}
Using the above, $\|(a_1, a_2)\|_p \geq 2^{\frac 1 p - \frac 1 2}\|(a_1, a_2)\|_2$ and $a_1^2 + a_2^2 \geq 1 - \frac{4c_2}{p} \geq 0.99$ for $p > 10^{56}$ we have
\begin{align}\label{thmaxr}
    C'_p\left( \frac{2^{\frac 2 p-1}}{\|(a_1, a_2)\|_p^2} - 1 \right) \leq C'_p\left( \frac{1}{a_1^2+a_2^2} - 1 \right) = \frac{\left(C'_p\right)^2\alpha^2}{a_1^2 + a_2^2} \leq 4.2\alpha^2.
\end{align}\\
To prove the inequality $\E\left( |X|^{-2} - \alpha^{-2} \right)_+ \ge 4.2\alpha^2$, we consider the event \\
$\mathcal{E} = \{R_1 \leq 1, |R_1 - R_2| < \alpha, |a_1\xi_1 + a_2\xi_2| < \frac{1}{4}\alpha\}$. On $\mathcal{E}$ we have
\begin{align*}
    |X| &= |a_1R_1\xi_1 + a_2R_2\xi_2| \leq |a_1R_1\xi_1 + a_2R_1\xi_2| + |a_2R_2\xi_2 - a_2R_1\xi_2| \\ &= R_1|a_1\xi_1 + a_2\xi_2| + a_2|R_2 - R_1| \leq \frac{1}{4}\alpha + 0.73\alpha = 0.98\alpha,
\end{align*}
where in the last inequality we used $a_2 \le \frac{1}{\sqrt{2}} + \frac{c_2}{p} \le 0.73$, which holds true for $p > 10^{43}$. Hence
\begin{align}\label{thmaxl}
    \E\left( |X|^{-2} - \alpha^{-2} \right)_+ &\geq \E\left( |X|^{-2} - \alpha^{-2} \right)\mathbf{1}_{\mathcal{E}} \geq 0.04\alpha^{-2}\P(\mathcal{E}) \notag \\ &= 0.04\alpha^{-2}\P(R_1 \leq 1, |R_1 - R_2| < \alpha)\P\left( |a_1\xi_1 + a_2\xi_2| < \frac{1}{4}\alpha \right) \notag \\ &=: 0.04\alpha^{-2}P_1P_2.
\end{align}\\
Treating $S^3$ as the uniform sphere in $\C^2$ and $\xi_1$, $\xi_2$ as $\C^2$-valued vectors, we obtain that $|a_1\xi_1 + a_2\xi_2|^2 = a_1^2 + a_2^2 + a_1a_2(\langle \xi_1, \xi_2 \rangle_{\C} + \langle \xi_2, \xi_1 \rangle_{\C})$ has the same distribution as $a_1^2 + a_2^2 +2a_1a_2\re D$ for $D \sim \text{Unif}(\D)$, since $\langle \xi_1, \xi_2 \rangle \sim \text{Unif}(\D)$ and $\langle \xi_2, \xi_1 \rangle = \overline{\langle \xi_1, \xi_2 \rangle}$. Thus we have
\begin{align*}
    P_2 = \P\left(\re D < \frac{\frac{\alpha^2}{16} - a_1^2 - a_2^2}{2a_1a_2}\right).
\end{align*}
Using \cite[Lemma 10]{ENT} with $c = c_2$ we obtain
\begin{align*}
    \frac{\frac{\alpha^2}{16} - a_1^2 - a_2^2}{2a_1a_2} &= \frac{\alpha^2}{32a_1a_2} - \frac{(a_1 - a_2)^2}{2a_1a_2} - 1 \geq \frac{\alpha^2}{32a_1a_2} - 1 - 13.3225\frac{c_2}{p-2}\left(C'_p\right)^2\alpha^2 \\ &\geq \frac{\alpha^2}{32\left( \frac{1}{\sqrt{2}} + \frac{c_2}{p} \right)^2} - 1 - \frac{100c_2\alpha^2}{p-2} \geq \frac{\alpha^2}{32} - 1,
\end{align*}
provided that $\frac{1}{\sqrt{2}} + \frac{c_2}{p} < 0.75$ and $p-2 > 3 \cdot 10^4c_2 = 6 \cdot 10^{45}$, which implies $\frac{1}{32\left(\frac{1}{\sqrt{2}} + \frac{c_2}{p}\right)^2} - \frac{100c_2}{p-2} \ge \frac{1}{32}$. By the above and since the density $h(t) = \frac{2}{\pi}\sqrt{1-t^2}$ of $\re D$ is concave on $[-1, 1]$, we have
\begin{align}\label{P2}
    P_2 &\geq \P\left( \re D < \frac{\alpha^2}{32} - 1 \right) = \int_{-1}^{\frac{\alpha^2}{32} - 1}\frac{2}{\pi}\sqrt{1-t^2}dt \geq \frac{1}{\pi}\sqrt{1 - \left( \frac{\alpha^2}{32} - 1 \right)^2}\frac{\alpha^2}{32} \notag \\ &= \frac{1}{\pi}\sqrt{\frac{\alpha^2}{16} - \frac{\alpha^4}{2^{10}}}\frac{\alpha^2}{32} \geq 2^{-9}\alpha^3,
\end{align}
as we have $\frac{1}{C'_p} \le 2^{\frac 2 p - 1} \le 0.51$ for $p > 100$. Hence $\alpha^2 \leq 0.51 \cdot \frac{4c_2}{p} \leq 10^{-3}$ for $p > 10^{56}$ and thus $\frac{\alpha^2}{16} - \frac{\alpha^4}{2^{10}} \geq \frac{\alpha^2}{16}\left(
1 - 40^{-3} \right)$.\\
Next we bound $P_1$ from below. Let $g(r) = c_p^{-1}r^{p+1}\exp(-r^p)$ for $r > 0$ be the density of $R_1$. We have $g(1) = \frac{p}{\Gamma\left( 1+\frac 2 p \right)}e^{-1} > \frac{p}{4}$ and
\begin{align*}
    g\left( 1 - \frac{1}{2p} \right) &= \frac{p}{\Gamma\left( 1+\frac 2 p \right)}\left( 1 - \frac{1}{2p} \right)^{p+1}e^{-\left( 1 - \frac{1}{2p} \right)^p} > \frac{p-1}{2}e^{-e^{-\frac{1}{2}}} > \frac{p}{4},
\end{align*}
where the first inequality follows from the inequalities $\frac{1}{\Gamma\left(1 + \frac 2 p\right)} > 1$ for $p > 2$, $p\left(1 - \frac{1}{2p}\right)^{p+1} \ge p\left(1 - \frac{p+1}{2p}\right) = \frac{p-1}{2}$ and $\left(1 - \frac{1}{2p}\right)^p < e^{-\frac 1 2}$, and the second inequality follows by $p > 25$ and the numerical approximation $e^{-e^{-\frac 1 2}} > 0.54$. Together with log-concavity of $g$ these lower bounds give $g(r) \geq \frac{p}{4}\mathbf{1}_{1-\frac{1}{2p} \leq r \leq 1}$. Hence
\begin{align}\label{P1}
    P_1 &\geq \int_{x \leq 1, |x - y| < \alpha}\frac{p^2}{16}\mathbf{1}_{\left[1-\frac{1}{2p}, 1\right] \times \left[1-\frac{1}{2p}, 1\right]}(x, y)dxdy = \begin{cases}
        \frac{1}{64} &\textrm{for } \alpha > \frac{1}{2p} \notag \\
        \frac{p^2\alpha}{16}\left(\frac{1}{p} - \alpha\right) &\textrm{for } \alpha \leq \frac{1}{2p}
    \end{cases} \\ &\geq 25 \cdot 2^9 \cdot 4.2\alpha,
\end{align}
where the last inequality uses the fact that $\frac{p^2}{16}\left(\frac{1}{p} - \alpha\right) \geq \frac{p}{32} \ge 25 \cdot 2^9 \cdot 4.2\alpha$ for $\alpha \leq \frac{1}{2p}$ and $\alpha \le \sqrt{0.51 \cdot \frac{4c_2}{p}} \le 10^{-7} \le \frac{1}{64 \cdot 25 \cdot 2^9 \cdot 4.2}$ for $p > 10^{56}$.

Putting \eqref{P2}, \eqref{P1} into \eqref{thmaxl} and using \eqref{thmaxr}, we conclude that \eqref{thmax} is satisfied. This ends the proof of Theorem \ref{th2}. \hfill $\Box$\\

\textbf{Acknowledgement.} We wish to thank Piotr Nayar for helpful discussions and the anonymous referee for questions and suggestions helping to improve the paper.

\vspace*{1cm}

\noindent Institute of Mathematics \\
University of Warsaw \\
02-097, Warsaw, Poland \\
jj406165@mimuw.edu.pl \\

\noindent Mathematisches Seminar \\
Universit\"at Kiel \\
24098 Kiel, Germany \\
hkoenig@math.uni-kiel.de \\

\end{document}